\documentclass[10pt]{article}
\font\smallit=cmti10
\font\smalltt=cmtt10

\usepackage{amssymb,latexsym,amsmath,epsfig,amsthm} %% Add other packages as necessary

\makeatletter

\renewcommand\section{\@startsection {section}{1}{\z@}
{-30pt \@plus -1ex \@minus -.2ex}
{2.3ex \@plus.2ex}
{\normalfont\normalsize\bfseries\boldmath}}

\renewcommand\subsection{\@startsection{subsection}{2}{\z@}
{-3.25ex\@plus -1ex \@minus -.2ex}
{1.5ex \@plus .2ex}
{\normalfont\normalsize\bfseries\boldmath}}

\renewcommand{\@seccntformat}[1]{\csname the#1\endcsname. }

\makeatother

\newtheorem{theorem}{Theorem}
\newtheorem{lemma}{Lemma}

\newtheorem{corollary}{Corollary}

\theoremstyle{definition}
\newtheorem{definition}{Definition}

\newcommand{\roller}{\circledcirc}
\newcommand{\N}{\mathcal{N}}
\renewcommand{\P}{\mathcal{P}}

\newcommand{\R}{\mathcal{R}}
\newcommand{\Le}{\mathcal{L}}
\begin{document}

\begin{center}
\uppercase{\bf \boldmath Mis\`ere \sc{cricket pitch}}
\vskip 20pt
{\bf Richard J. Nowakowski\footnote{Research supported by}}\\
{\smallit Department of Mathematics \& Statistics, Dalhousie University, NS, Canada}\\
{\tt r.nowakowski@dal.ca}\\ 
\vskip 10pt
{\bf Ethan J. Saunders\footnote{Research supported by}}\\
{\smallit Department of Mathematics \& Statistics, Dalhousie University, NS, Canada}\\
{\tt ejs062003@gmail.com}\\ 
\end{center}
\vskip 20pt
\centerline{\smallit Received: , Revised: , Accepted: , Published: } % We will fill in the dates
\vskip 30pt

\centerline{\bf Abstract}
\noindent
Mis\`ere games in general have little algebraic structure, but if the games under consideration have properties then some algebraic structure re-appears. In 2023, the class of Blocking games was identified. Mis\`ere \textsc{cricket pitch} was suggested as a problem at the Games-at-Dal-2023 Workshop, and is the first game in this class to be studied.  Normal play \textsc{cricket pitch} was analyzed in 2011.  

The game involves flattening `bumps' with a roller. The main reduction of normal play, reducing every bump by $2$, is not applicable in mis\`ere play.
 In this paper, we find the outcomes of single (linear) component \textsc{cricket pitch} positions, where
 the proof is based on first considering the bumps to the left and to the right of the roller separately. We also give reductions, true in Blocking games in general, of positions that occur in simple positions of \textsc{cricket pitch}. These allow us to find the outcomes of the disjunctive sum of single bump positions. At the time of writing, it was not possible to find game values since the relevant theory for Blocking games does not exist. 

\pagestyle{myheadings}
\markright{\smalltt INTEGERS: 24 (2024)\hfill}
\thispagestyle{empty}
\baselineskip=12.875pt
\vskip 30pt

\newcommand{\carlos}[1]{{\bf [Carlos:} {\magenta #1}{\bf ]}}
\newcommand{\rjn}[1]{{\bf [Richard:} {\blue #1}{\bf ]}}
\newcommand{\Neil}[1]{{\bf [Neil:} {\red #1}{\bf ]}}
\newcommand{\Reb}[1]{{\bf [Rebecca:} {\grey #1}{\bf ]}}

\tolerance=1 %Linebreak Avoidance
\emergencystretch=\maxdimen
\hyphenpenalty=10000
\hbadness=10000

\section{Introduction}\label{Rules}
Taking care of a real-life cricket pitch is a an unenviable task (\cite{Manual}). One important task is ensuring the surface is flat. A roller (possibly weighing several tons) is used to flatten the bumps, and a flat section should not be rolled again, if possible. A simplified game version of caring for a pitch was introduced in \cite{NowakO2011}. It was noted that the game was an instance of a larger class of games and the values  of the normal play version was solved, to within an infinitesimal. In this paper, we start the analysis of \textsc{cricket pitch}, and the larger class, under the mis\`ere winning convention.\\

\noindent \textit{Rules.}  A \textsc{cricket pitch} board is a sequence of bumps (positive integers) and one roller situated between two bumps or is at one end.
Left (Right) moves the roller to the left (right) any positive number of bumps.  Each
bump that is rolled over is reduced by $1$.  A value of $0$ is no longer a bump, and
can no longer be rolled over. \\

For example, with Left moving first (not best play):
\[2\, 3\, 2\roller4\,2\,\frac{L}{}\,2\roller2\,1\,4\,2\,\frac{R}{}\,2\,1\,0\roller4\,2=\roller4\,2\]
would have Left losing the Normal play game and winning the mis\`ere play game.

\textsc{cricket pitch} has the property that  if a player moves twice, the resulting position could have been achieved in one move. That is
$G^{\Le\Le}\subseteq G^\Le$ and $G^{\R\R}\subseteq G^\R$.
These types of games are called \textit{hereditarily closed} (originally option-closed). For mis\`ere this class has been extended to \textit{Blocking games}, $B$: if, say Left, has no moves in $G$ then, in every $G^R$ either Left still has no move, or there is $G^{RL}$ in which Left has no move.
In \textsc{cricket pitch} this would be accomplished by moving the roller to the end.

For normal play, in \cite{NowakO2011}, it is shown that a single pitch has the value $\{a|b\}+ \text{infinitesimal}$,
where $a$ and $b$ are dyadic rationals.  An important step is that subtracting $2$ from every bump (providing no bump goes negative) does not change the value. The disjunctive sum of \textsc{cricket pitch} positions is now easy to analyze. Mis\`ere play is not so easy.  In mis\`ere, this reduction does not preserve the value. For instance, Left, moving first, wins $2,1\roller1,2 + 2\roller2,1$ but loses $4,3\roller3,4 + 2\roller2,1$.
 In this paper, Theorems \ref{thm:1st pitch}, \ref{thm:2nd pitch} give the outcomes for one-pitch \textsc{cricket pitch}. Lemma \ref{lem:reductions} give some reductions that are true for the Blocking universe not just \textsc{cricket pitch}. 
% These reductions allow us to find the outcomes of disjunctive sum of pitches that only have one bump, Theorem \ref{thm:disj sum}, and some that have two, Theorem \ref{thm:2bump}.  
The reductions require the following standard theory: for positions $G, H\in B$

\begin{enumerate}
\item $o(G)$ is the outcome of $G$, and $o_L(G)$ ($o_R(G)$) is the outcome of $G$ with Left (Right) playing first. Note, both $o_L(G)$ and $o_R(G)$ are either a left-win or a right-win.
\item For a universe $U$, and $G,H\in U$, $G\equiv_U H$ if $o(G+X)=o(H+X)$ for all $X\in U$.
\end{enumerate}
 
 \noindent Following the conventions introduced by Conway ("On Numbers and Games" ), and Berlekamp, Conway and Guy ("Winning Ways"), when the context is clear, $G\equiv_U H$ is replaced by $G=H$.\\
 
\subsection{The Larger Context}
The Blocking universe\footnote{A \textit{universe} is a set of games $U$ such that if $G\in U$ then (i) all options of $G$ are in $U$; (ii) $-G\in U$; (iii) if $H\in U$ then $G+H\in U$; and (iv) if both $G$ and $H$ are non-empty then $\{G\mid H\}\in U$.} is a superset of the other studied universes, Dicotic and Dead-ending.

As examples of some results, the following are easy to prove after reading the next two sections, especially Lemma \ref{lem:reductions}:
\begin{itemize}
\item  $1\roller+\roller1 = 0$; $2n\roller = 0$;
\item   $1\roller1\ne0$; $2n\roller2n\ne0$; 
\item $1\roller1+1\roller1\ne0$;
 $2n\roller2n+2n\roller2n\ne0$; $2,1\roller+\roller1,2\ne0$;
\end{itemize}

In mis\`ere universes, the Conjugate Conjecture states that if $G+H=0$ then $H$ is $G$ with the roles of Left and Right interchanged. This is often called `turning the board around'. This is true in normal play. It is not necessarily true in non-universes, for example, Milley\cite{Mill}. The last line of the previous examples, show positions that are most probably zugzwang positions, i.e., positions without inverses and in which neither player wants to move.

This section concludes with the obligatory analysis: in both
\[1,1,6,2,4,5\roller4,3,3,4,6\quad\text{and}\quad 6,4,2,1\roller1,3,5,7,8\]
find who wins and how?

\section{Outcomes}\label{Outcomes}
We will be considering $G$ as a position of the form $a_{m}a_{m-1}\ldots a_{1}\roller b_1b_2\ldots b_n$, where 
$\roller$ is the roller, and the size of the bumps are given by $a_i$ and $b_j$. We will call the bumps to the left of the roller the \textbf{left}-side and the others the \textbf{right}-side. We will refer to the $i$-th bump, the context will make it clear if it is on the left or right. In our analysis, as play continues, we keep the original left-and right-sides, they will not be updated unless explicitly redefined. To avoid writing all the subscripts every time, we will write $G=\alpha\roller\beta$, with the subscripts implied. Any result mentioning one player, is  also true for the other and usually will not be explicitly stated.

\begin{lemma}\label{lem:one-side}
Let $G = \alpha\roller$ where $\alpha = a_0,\dots,a_k$ is a nonempty string of positive values. Then, 
$o(G)=\N$ if and only if $\alpha$ contains at least one even, non-zero number, otherwise, $o(G)=\R$. 
\end{lemma}
\begin{proof} In $G = \alpha\roller$, Right always wins moving first. Now consider Left moving first. 

Suppose $\alpha$ consists of odd numbers. Clearly, $o(1\roller)=\R$.  For any other $\alpha$,  Right responds to Left's move by moving the roller as far to the right as possible (an empty bump may have been created by  Left's move). This results in a position $\beta\roller$ where $\beta$ consists of odd numbers, and Right wins by induction.

Suppose  $G=\beta,2k\roller$. Left moves the roller one position, forcing Right to move to  
$\beta,(2k-2)\roller$ and Left wins by induction. Now suppose $G=\beta,2k,\gamma,(2j+1)\roller$ where 
$\gamma$ consists of odd numbers.  
Left moves the roller over all the bumps of size $1$ in $\gamma$, if any, and then one further bump. The position is
either (i) $\beta,2k,\gamma'\roller(2j')$, where, now, all the bumps in $\gamma'$ are odd and at least $3$, or (ii) $\beta\roller(2k-1)$. 
In the first case, Right can only move to 
  $\beta,2k,\gamma',(2j'-1)\roller$, and in the second, Right can only move
   to $\beta,(2k-2)\roller$. In both cases, Left (moving first) wins by induction. 
\end{proof}

This leads to a reduction for outcomes.
\begin{lemma}\label{lem:removeodd}
Let $G=\alpha\roller\beta,(2d+1)$ with $\beta$ non-empty, then $o(G) =o(\alpha\roller \beta)$.
\end{lemma}
\begin{proof} Note that $o(\roller)\ne o(\roller 1)$, since $\roller$ is a next player win and $\roller1$ is a Left win. Thus $\beta$ cannot be empty. If, on any move, Right moves the roller to the end, then Left can regard the game as being $(2d)\roller$.  After $2d$ moves, it is Left to move in $0\roller$ which she wins. Thus, if Right wins $G$ playing first or second, he never plays to the end, which proves the result.
\end{proof}
This reduction is not necessarily true for values, and cannot be used in the disjunctive sum of \textsc{cricket pitch} positions.
\begin{corollary}[Odd tails are removable]
Let $G=\alpha\roller\beta,\gamma$ where every element of $\gamma$ is odd.
 Then $o(G) =o(\alpha\roller \beta)$.
\end{corollary}
\begin{proof}
Repeatedly apply Lemma \ref{lem:removeodd}.
\end{proof}

A position $G=\alpha\roller\beta$ is \textit{reduced} if each sequence is either empty or ends with an even
number.

Lemma \ref{lem:one-side} gives us the outcomes of $\alpha\roller$ and $\roller\beta$. Combining these gives four possibilities for $G=\alpha\roller\beta$, and three are easy to analyze.

\begin{theorem}\label{thm:1st pitch}
Let $G$ be a game of the form $\alpha\roller\beta$, where either sequence could be empty. 
\begin{enumerate}
\item If $(o(\alpha\roller),o(\roller\beta)) = (R,L)$ then $o(G)=P$.
\item If $(o(\alpha\roller),o(\roller\beta))= (R,N)$ then $o(G)=R$.
\item If $(o(\alpha\roller),o(\roller\beta))= (N,L)$ then $o(G)=L$.
\end{enumerate}
\end{theorem}

\begin{proof}
Suppose $(o(\alpha\roller),o(\roller\beta)) = (R,L)$. When  Left moves first, Right's strategy is to avoid moving the roller into the right-side. This way, the entire game will take place within the left-side. 
Since $\alpha\roller$ is a Right win, we know Right will eventually be in a position where he has to move and the roller is $\ldots a_{-k}\roller 0_{-k+1}\ldots$. That is, Right has won $G$. If Right moves first, Left uses the same strategy.

Suppose $(o(\alpha\roller),o(\roller\beta)) = (R,N)$.
Suppose Left moves first then Right wins by applying the same strategy as in the previous case.
Now, suppose Right moves first. His strategy is to play a winning move in $\roller\beta$. If Left responds in the right-side (i.e. the roller stays to the left of its initial position) then, Right continues to play winning moves in the right-side. Either the game ends and Right wins, or Left moves the roller into the
left-side. Now Right can play as in the previous paragraph after Left move first.

The case $(o(\alpha\roller),o(\roller\beta)) = (N,L)$ is analogous to the previous case.
\end{proof}

The remaining case splits into sub-cases and is reminiscent of the normal-play analysis.

\subsection{Analysis of $(\N,\N)$}

We now assume that $G=\alpha\roller\beta$ and $o(\alpha\roller)=o(\roller\beta)=N$. Intuitively,
our approach is: each player on their first move can move into `their' component and make the appropriate winning move. The component that finishes first will determine the winner. For example, if it is the left-side then Right's last move must be in $\ldots\roller1_{-k},0_{-k+1}\ldots$. Since the $1,0$ bumps have been rolled an odd number of times, initially, they were even and odd respectively. Since there could be such a pair in the right-side, we must determine which pair reaches the end first. In addition, we have to take into account bumps between the odd-even pairs that are eliminated before one of the pairs reaches the $0,1$ stage.

In what remains, the least odd number that cannot be isolated by the opponent is of importance.

\begin{definition} Let $\gamma=a_1,a_2,\ldots,a_s$ be a sequence of positive integers. Let
 $m(\gamma)=\max\{i: \text{ $a_i$ is odd, and } a_i\leq a_j, j<i, \}$, and, if $m(\gamma)$ exists, let $M(\gamma)=a_{m(\gamma)}$. Otherwise, $m(\gamma)=M(\gamma)=\infty.$ 
  \end{definition}

Recall that in $G=\alpha\roller\beta$ the sides are indexed from the roller out. 
For example, if $G=1,2,1,5,3\roller5,3,2,3$, then $m(\alpha)=5$, $M(\alpha)= 1$, $m(\beta)=2$, and $M(\beta)=3$, and 
if $G=2,3,1,2\roller2,3$  then $m(\alpha)=2$, $M(\alpha)= 1$, and $m(\beta)=M(\beta)=\infty$. 
The bumps $a_{m(\alpha)}$ and $b_{m(\beta)}$ are important, in that the players want to push the roller that far and no further.
\begin{theorem} \label{thm:2nd pitch}
Let $G=\alpha\roller\beta$ be a reduced \textsc{cricket pitch} position where $\alpha$ and $\beta$ are  non-empty. 
\begin{enumerate}
\item If $M(\alpha) < M(\beta)$, then $o(G)=\Le$.
\item If $M(\alpha)  > M(\beta)$, then $o(G)=\R$.
\item If $M(\alpha)  = M(\beta)<\infty$, then $o(G)=\N$.
\item If $M(\alpha)  = M(\beta)=\infty$, then $o(G)=\P$.
\end{enumerate}
\end{theorem}

\begin{proof}
Since $G$ is reduced, then the last bumps in $\alpha$ and $\beta$ are even. Let $\alpha = a_p,a_{p-1},\ldots,a_1$ and, for brevity, let $\ell = m(\alpha)$ and $\alpha' = a_p,a_{p-1},\ldots,a_{\ell+2}$. Note that 
$a_{\ell+1}$ exists since $G$ is reduced ($\alpha$ doesn't end on an odd number) and $a_\ell$ is odd. Let $r=m(\beta)$. One strategy is central to this proof.

\begin{definition} 
With Left to move, her \textit{basic} strategy is:

if $a_\ell>1$, then move to $\alpha',a_{\ell+1}\roller (a_{\ell}-1),\ldots,a_0-1,\beta$; 

if $a_\ell=1$, then move to $\alpha',a_{\ell+2}\roller (a_{\ell+1}-1),0$.

Right's basic strategy is analogous.
\end{definition}
\noindent
Parts (1,2 3): Suppose $m(\alpha)<\infty$.

Let $s=\min\{i: a_i\text{ is even, and }i>\ell\}$ and $t=\min\{j: b_j\text{ is even, and }j>r\}$. 
Because $G$ is reduced, and if $m(\alpha)$ and $m(\beta)$ are not infinite, then $s$ and $t$ exist.

Suppose $a_\ell\leq b_r$, $\ell<\infty$, and Left moves first. We assume that Left always attempts to play her basic strategy. If she can, then she eventually plays to $\alpha',a_{\ell+2},\roller (a_{\ell+1}-1)$ with Right to move. Left now  moves the roller over only the $\ell+1$ bump. This leads to the position $\alpha',a_{\ell+2},0,\roller$ and Left to move, i.e., she wins.

If, on some Left move, Left could not play her basic strategy, the
position would be of the form $\gamma_1,a'_\ell,\gamma_2,0,\gamma_3\roller\delta$, and $a'_\ell\leq a_\ell$.
 The $0$-bump cannot be in the left-side (i.e. $\alpha$) by the choice of $\ell$. Therefore, it is on the right-side and must have started as an odd sized bump, say $b_k$. Now, $a_\ell>b_k$ because, if Left had played $a_\ell$ moves (using the basic strategy), the $a_\ell$ bump would have been reduced to $0$ and the roller would be to the left of that zero bump. Also, $b_k$ is
 less than $b_i$, $i=0,1,\ldots k-1$, since each of those bumps is reduced by $1$ whenever the $k$th bump is. By definition, $b_r\leq b_k$, thus $b_r<a_\ell$ which contradicts our assumption.  Note that this does not depend on $r<\infty$.

To recap, if $a_\ell\leq b_r$, and $\ell<\infty$, then Left wins playing first. 
 
Suppose $a_\ell<b_r$ and Right moves first. Let $(\gamma-2)$ represent the sequence $\gamma$ with every number reduced by $2$.
Left's response is in two parts, first move the roller back to the original position. This is possible since $b_r>1$, every bump in $\beta$ is at least $2$. 
At this intermediate point, the position is $\alpha\roller(\beta_1-2)\beta_2$,
 where $\beta_1\beta_2=\beta$ and it is Left to move. In this new position, $m(\alpha)=\ell$, $M(\alpha)=a_\ell$,
  and $M((\beta_1-2)\beta_2)\geq b_r-2$. However, $a_\ell<b_r$, and both are odd, therefore
  $a_\ell\leq b_r-2\leq M((\beta_1-2)\beta_2)$. The previous argument shows that Left wins, and she completes her move using that strategy.

This proves cases $1$, $2$, and $3$.\\

\noindent
Part (4): By Lemma \ref{lem:one-side} we may assume that there are non-zero even-sized bumps on both sides of the roller.   Since $m(\alpha)=m(\beta)=\infty$, then both $a_1$ and $b_1$ are even.

\textit{Claim:} If $m(\alpha)=m(\beta)=\infty$, and, for some $i$, $a_i$ is odd, then there is  $p$, $i>p$ such that $a_p$ is even and $a_i>a_p$.

\textit{Proof:} Again, note that $a_1$ is even. The proof is by induction. By the definition of $m(\alpha)$, there is $j<i$ such that $a_j<a_i$. If $a_j$ is even, then the statement is true. If $a_j$ is odd, then, by induction, there exists $p$,  $j>p\geq 1$ with $a_j>a_p$, and $a_p$ even. In conclusion, $a_i>a_j>a_p$ and $i\geq 1$, and the Claim is proved.\\

Now, consider Left moving first, moves over $k$ bumps, for some $k\geq 1$. 
Now $G=\alpha'\gamma\roller\beta$, where $\gamma=a_k,a_{k-1},\ldots,a_1$ and $\alpha=\alpha'\gamma$
Let $G''=\roller\delta$, where $\delta=(\gamma-1)\beta$. If $a_k$ is even, then 
$M(\delta)$ exists since $a_k-1$ is a candidate. If $a_k$ is odd, then by the Claim there is an even-sized 
bump in $a_{k-1},\ldots,a_1$, smaller than $a_k$. Let $a_s$ be the least sized even bump in $a_{k-1},\ldots,a_1$. In particular, $a_s$ is also less than any odd-sized  bump in $a_k,\ldots,a_s$. Therefore, $a_s-1$
is a candidate for $M(\delta)$. Since $m(\beta)=\infty$, then $M(\delta)=a_s$.
In either case $M(\delta)<\infty$.

 If $m(\alpha')=\infty$, then Right wins by part $2$ of the Theorem. Suppose $m(\alpha')<\infty$, and $M(\alpha')=a_t$. By definition, every even-sized bump in $a_{t-1},\ldots,a_{k+1}$ is greater than $a_t$.
 However, since $m(\alpha)=\infty$ then, there is some even-sized bump in $\gamma$ smaller than $a_t$, in particular  $a_s<a_t$. Therefore, $M(\alpha')>M(\delta)$, with Right to play. By part $2$, Right wins.

The argument for Right going first is analogous, thus $o(G)=\P$.
\end{proof}

%%%%%%%%%%%%%%%%%%%%%%%%%%%%%%%%%%%%%%
\section{Disjunctive sums of one-bump positions}\label{One-sidegames}
%%%%%%%%%%%%%%%%%%%%%%%%%%%%%%%%%%%%%%

Although the next Lemma will only be applied for \textsc{cricket pitch}, we prove it in the more general context
of Blocking games. This allows us to apply them in the disjunctive sum of\textsc{cricket pitch} positions.
%%%%%%%%%%%%%%%%%%%%%
\begin{lemma}\label{lem:reductions}[Reductions] Let $d$ and $e$ be non-negative integers, respectively, odd and even.
Within the blocking universe,  $B$,  
\begin{enumerate}
\item $e\roller=0$;
\item $d\roller=1\roller$;
\item $\roller1+1\roller = 0$;
\end{enumerate}
\end{lemma}
\begin{proof} Part ($1$): Recall that $o_L(G), o_R(G)$ are the outcomes when Left, respectively, Right, moves first.

We first show that $o(G+e\roller) \geq o(G)$. First, the situation when Left moves first.
Suppose $o_L(G)=R$ then automatically $o_L(G+e\roller) \geq o_L(G)$. Suppose $o_L(G)=L$, that is, she has a winning move to some $G^L$, or there are no left options. If there is a winning left option $G^L$, then in
$G+e\roller$, we claim that $G^L+e\roller$ is a winning move. Right must move
 to some $G^{LR}+e\roller$, but then $o_L(G^{LR})=L$ and Left wins by induction. If Left has no option in $G$, then, in $G+e\roller$, we claim Left wins by moving to $G+\roller(e-1)$. If Right plays to $G+(e-2)\roller$ then Left wins by induction. If Right plays to $G^R+e\roller$, then since Left had no move in $G$, in $G^R$ she has a move to some $G^{RL}$ in which she has no options. Since $o_L(G^{RL})=L$, Left wins $G^{RL}+e\roller$
 by induction.
 
Consider Right moving first. If $o_R(G)=R$ then automatically $o_R(G+e\roller) \geq o_R(G)$.
Suppose that $o_R(G)=L$, and consequently, Right has an option. In $G+e\roller$, Right moves to some $G^R+e\roller$. In $G^R$, suppose Left has a winning move to some $G^{RL}$. Thus we are in the position 
$G^{RL}+e\roller$ with $o_R(G^{RL})=L$ which Left wins by induction. If, in $G^R$ Left has no move, that is, $o_L(G^R)=L$, then by the first part, $o_L(G^R+e\roller)=L$.

Thus, $o(G+e\roller) \geq o(G)$.

Now, we show that $o(G)\geq o(G+e\roller)$. Suppose $o_L(G+e\roller)=R$ then automatically 
$o(G)\geq o(G+e\roller)$.  

If $o_L(G)=R$  then Right wins $G+e\roller$ going second: (i) if $G^L$ has no right option then $G^L+e\roller$ has no right option; (ii) by moving to either $G+(e-2)\roller$ (wins by induction) or to $G^{LR}+e\roller$
where $G^{LR}$ is Right's winning move in $G^L$, and again Right wins by induction. Therefore, $o_L(G)=L\geq o(G+e\roller)$.

If $o_R(G)=R$, then there is nothing to prove. Suppose $o_R(G)=L$. In $G+e\roller$, Right plays to 
$G^R+e\roller$. In $G^R$, Left has a winning move to $G^{RL}$, so she responds with
$G^{RL}+e\roller$ and wins by induction.

Thus $o(G)\geq o(G+e\roller)$, proving $e\roller\equiv_B0$.
%%%%%%%%%%%%%%%%%%%%%%%%%%%%%%%

Part ($2$): Note that
\begin{eqnarray*}
d\roller&=&\{ \roller(d-1)\mid\,\}\\
&=&\{ 0\mid  \},\quad\text{by part 1,}\\
&=&1\roller.
\end{eqnarray*}
%%%%%%%%%%%%%%%%%%%%%%%%%%%%
Part ($3$): $\roller1+1\roller= 0$.

We first show that $o(\roller1+1\roller+X)\geq o(X)$. If $o_L(X)=R$ then there is nothing to prove.
Suppose  $o_L(X)=L$ and consider Left playing first in $roller1+1\roller+X$. If this is to  
$\roller1+1\roller+X^L$ then Right has a winning move, $X^{LR}$, in $X^L$. Therefore, Right playing to
$\roller1+1\roller+X^{LR}$ wins by induction. If Right has no move in $X^L$, then he moves to $1\roller +X^L$.
If Left now plays to $X^L$, Right wins. If Left plays to $1\roller +X^{LL}$, then either Right has no move and wins, or plays to $1\roller +X^{LLR}$, in which he has no move. Eventually, Left has to play to $\roller +X'$
in which Right has no move.

The other cases are similar and are left to the reader.
\end{proof}
%%%%%%%%%%%%%%%%%%%%%%%%%%%%%%

Finding the outcome of a disjunctive of one-bump positions is now straightforward. Applying, Lemma \ref{lem:reductions}, will show that at most one player has a move.

\begin{theorem}\label{thm:disj sum} Let $G = \sum_{i=1}^n G_i$, where $G_i$ is a one-bump position. Let $\ell$ be the number of Left-win components, and $r$ be the number of Right win. 
\begin{eqnarray*}
o(G) =\begin{cases} L\text{ if $\ell>r$;}\\
N\text{ if $\ell=s$; and}\\
R\text{ if $\ell<r$;}
\end{cases}
\end{eqnarray*}
\end{theorem}
\begin{proof}
By Lemma \ref{lem:reductions}, all components
reduce to one of $\roller1$, $1\roller$, or $0$, which are Right, Left, and Next wins respectively.
Thus $\ell$ counts the number of $1\roller$ components and $r$ the number of $\roller1$. The result now follows by applying the third part of 
Lemma \ref{lem:reductions}. 
\end{proof}

\section{Comments}
Let us answer the question from the Introduction. The reduced position of $1,1,6,2,4,5\roller4,3,3,4,6$ is $6,2,4,5\roller4,3,3,4,6$, and $M(\alpha)=5$ and
$M(\beta)=3$. Thus Right wins by moving just past $b_3$. The position of $6,4,2,1\roller2,3,5,7,8$ is already reduced, and $M(\alpha)=1$, $M(\beta)=\infty$. Left wins by playing just past $a_2$. 

A general solution of the disjunctive sum of \textsc{cricket pitch} positions would also require 
most of the general theory of Blocking games to be developed. However, disjunctive sums of positions
of the form $2,\alpha\roller\beta,2$ where $\alpha$ and $\beta$ consist of $1$s may be tractable. 

The `two moves is not better than one' property of \textit{cricket pitch} means that the game is still interesting when played under different terminating conditions, i.e. the other universes. The reader may wish to try solving \textsc{cricket pitch} in:

\textit{Impartial play:} Both players can move the roller to the left and to the right.

\textit{Dicotic play:} Either both players have a move or the game is over. For a single pitch, moving the roller to an end finishes the game.

\textit{Dead-ending play:} Once a player, say Left, has no moves, no move by Right will create moves for Left. In \textsc{cricket pitch}, if the roller stops at an end (or next to a $0$), then the roller is only allowed to move away from that end (or $0$).

See \cite{Dorb013,Fish022,Lar021,Lar022} for more on these universes.

\end{document}